\documentclass[a4paper, 11pt]{article}

\input xy

\xyoption{all}

\usepackage{amsmath,amsthm,amssymb}
\usepackage[page,header]{appendix}
\usepackage{morefloats}
\usepackage{color}
\usepackage{makeidx}
\usepackage{hyperref}	 	
\usepackage{fancybox}
\usepackage{mathabx}
\usepackage{tikz}
\usetikzlibrary{matrix}
\usepackage{arcs}
\usepackage{hyperref}

\makeindex 

\usepackage{latexsym, amsmath, amssymb, amsfonts, amscd}
\usepackage{amsthm}
\usepackage{t1enc}
\usepackage[mathscr]{eucal}
\usepackage{indentfirst}
\usepackage{graphicx, pb-diagram}
\usepackage{enumerate}
\usepackage[all,poly,web,knot]{xy}
\usepackage{float}
\restylefloat{figure}

\newcommand{\Title}{Title}

\numberwithin{equation}{section}

{\theoremstyle{definition}\newtheorem{definition}{Definition}[section]

\newtheorem{defititle}[definition]{\Title}

\newtheorem{terminology}[definition]{Terminology}
\newtheorem{remark}[definition]{Remark}
\newtheorem{remarks}[definition]{Remarks}
\newtheorem{ex}[definition]{Example}
\newtheorem{exs}[definition]{Examples}}
\newtheorem{prop}[definition]{Proposition}
\newtheorem{proposition-definition}[definition]{Proposition-Definition}

\newtheorem{thm}[definition]{Theorem}
\newtheorem{cor}[definition]{Corollary}

\newtheorem{assumptions}[definition]{Assumptions}

\newtheorem*{prop*}{Proposition}
\newtheorem*{theorem*}{Theorem}

\newcommand{\cG}{\mathcal{G}}

\newcommand{\cF}{\mathcal{F}}
\newcommand{\cE}{\mathcal{E}}

\newcommand{\cJ}{\mathcal{J}}

\newcommand{\id}{{\hbox{id}}}

\newcommand{\eg}{{\it e.g.}\/ }
\newcommand{\cf}{{\it cf.}\/ }

\def\gpd{\,\lower1pt\hbox{$\longrightarrow$}\hskip-.24in\raise2pt
             \hbox{$\longrightarrow$}\,}

\renewcommand{\latticebody}{\drop@{ }}

\newcommand{\lra}{\longrightarrow}

\newcommand{\Z}{\ensuremath{\mathbb Z}}

\newcommand{\R}{\ensuremath{\mathbb R}}

\newcommand{\g}{\ensuremath{\mathfrak{g}}}

\newcommand{\cC}{\mathcal{C}}            

\def\act{\mathbin{\hbox{$<\kern-.4em\mapstochar\kern.4em$}}}
\def\ract{\mathbin{\hbox{$\mapstochar\kern-.3em>$}}}
\def\exp{\mathrm{exp}}

\def\PB(#1,#2,#3,#4){\left\{\begin{matrix}#1&\!\!\!\stackrel{?}{\longrightarrow}&\!\!\!#2\\
\downarrow&&\!\!\!\downarrow\\
#3&\!\!\!\stackrel{?}{\longrightarrow}&\!\!\!#4\end{matrix}\right\}}

\def\pb(#1,#2,#3,#4){ \hom(#1 \to #3, #2 \to #4)}

\begin{document}

\begin{center}
{\Large\bf Integrable lifts for transitive Lie algebroids\footnote{AMS subject classification: 	58H05 ~ Secondary 14F40. Keywords: Transitive Lie algebroid, integrable lift.} 

\bigskip

{\sc by Iakovos Androulidakis and
 Paolo Antonini 
}
}
 
\end{center}

{\footnotesize
Department of Mathematics 
\vskip -4pt National and Kapodistrian University of Athens
\vskip -4pt Panepistimiopolis
\vskip -4pt GR-15784 Athens, Greece
\vskip -4pt e-mail: \texttt{iandroul@math.uoa.gr}
  
\vskip 2pt SISSA, 
\vskip-4pt  Via Bonomea 265, 34136 Trieste, Italy
\vskip-4pt e-mail: \texttt{pantonin@sissa.it}

}
\bigskip
\everymath={\displaystyle}

\date{today}

\begin{abstract}\noindent 
Inspired by the work of Molino, we show that the integrability obstruction for transitive Lie algebroids can be made to vanish by adding extra dimensions. In particular, we prove that the Weinstein groupoid of a non-integrable transitive and abelian Lie algebroid, is the quotient of a finite dimensional Lie groupoid. Two constructions as such are given: First, explaining the counterexample to integrability given by Almeida and Molino, we see that it can be generalized to the construction of an ``Almeida-Molino'' integrable lift when the base manifold is simply connected. On the other hand, we notice that the classical de Rham isomorphism provides a universal integrable algebroid. Using it we construct a ``de Rham'' integrable lift for any given transitive Abelian Lie algebroid.
\end{abstract}
 
\setcounter{tocdepth}{2} 


\tableofcontents

\section*{Introduction}\label{sec:intro}

In this paper we revisit Lie's 3rd theorem, that is the question of integrability, in the more general context of Lie algebroids. We are inspired by the first example of Lie algebroids whose integrability fails, given by Almeida and Molino \cite{AlmeidaMolino} in their work concerning foliation theory. We take the point of view that adding extra dimensions can treat the integrability problem, especially if one is interested in applications to index theory and non commutative geometry. The first step in this direction is the examination of \textit{transitive} and \textit{abelian} Lie algebroids.

In fact we are motivated further by foliation theory, in particular by the following, constructive proof of Lie's 3rd theorem found in \cite{DuistKolk}: Given a finite dimensional Lie algebra $\g$, the Banach manifold $P(\g) = \Omega^1([0,1],\g)$ admits an action by gauge transformations. The orbits form a foliation with finite codimension and the associated leaf space $G(\g)$ turns out to be the connected and simply connected Lie group whose Lie algebra is $\g$. The deeper reason for the smothness of $G(\g)$ is really the vanishing of the second homotopy group for (finite dimensional) Lie groups.

The above construction was first generalised to Lie algebroids arising in Poisson geometry by Cattaneo and Felder \cite {CatFeld} and then to general Lie algebroids by Crainic and Fernandes \cite{CFLie}. However, Lie groupoids in general do not have vanishing second fundamental group (\cf \cite{TsengZhu}), so for a generic Lie algebroid $B \to M$, the leaf space $G(B)$ arising from the above construction is not necessarily a smooth topological groupoid, with connected and simply connected $s$-fibers. In \cite{CFLie} the explicit obstructions to the smoothness of $G(B)$ were given. Particularly in the transitive case\footnote{In the transitive case the integrability obstruction was first given by Mackenzie (see \cite{MK2}). However, in this paper we will use the formulation provided by Crainic and Fernandes \cite{CFLie}.}, the only obstruction is the discreteness of the image of a certain monodromy, boundary map, which is defined on $\pi_2(M)$ and takes its values in the isotropy of $B$. 

For a \emph{transitive} Lie algebroid $B$ such that this discreteness fails, the following problems arise: 
\begin{itemize}
\item In order to continue doing differential geometry, one works with the Banach manifold $P(B)$ (\cf \cite{CatContr}). But as $P(B)$ is infinite dimensional, the use of quite involved techniques is required.
\item The lack of smooth densities in $G(B)$ and $P(B)$ prevents using the standard techniques developed by A. Connes in order to attach a $C^{\ast}$-algebra, a pseudodifferential calculus and to deploy $K$-theoretic techniques. In particular, one wishes to be able to develop index theory when $B$ is not integrable. 
\item On simply connected manifolds, for those transitive Lie algebroids which arise from a symplectic form, the non-smoothness of $G(B)$ corresponds to the non-existence of a prequantizing principal $S^1$-bundle (\cf \cite{Crainic:2004aa}).
\end{itemize}
\paragraph{The lifting problem.} For all these reasons it would be preferable if $G(B)$ could be realized as the leaf space of a foliation on a \emph{finite dimensional} Lie groupoid $G(\widehat{B})$. Explicitely, one wishes to show that, given a non-integrable transitive Lie algebroid $B$, there exists a canonical transitive Lie algebroid $\widehat{B}$  such that: 
\begin{enumerate}
\item The monodromy group of $\widehat{B}$ is discrete, so $G(\widehat{B})$ is a Lie groupoid (finite dimensional).
\item There exists a surjective morphism of Lie algebroids $\widehat{B} \to B$. This integrates to a morphism of topological groupoids $G(\widehat{B}) \to G(B)$ which is onto.
\end{enumerate}
We refer to these questions as the \emph{lifting problem} for the (non-integrable) Lie algebroid $B \to M$. A solution is a Lie algebroid $\widehat{B}$ with the above properties, and its existence has the following implications: 
\begin{itemize}
\item $G(B)$ is realized as the leaf space of a foliation on the \emph{finite dimensional} Lie groupoid $G(\widehat{B})$, by the orbits of the action of $\ker(\widehat{B} \to B)$. 
\item The Lie groupoid $G(\widehat{B})$ does have $C^{\ast}$-algebras and the fibration $G(\widehat{B}) \to G(B)$ conjecturally provides a Thom isomorphism in $K$-theory. 
\item Regarding geometric prequantization, although a smooth principal $S^1$-bundle does not exist, a solution to the lifting problem provides the quotient of a smooth principal bundle instead (\cf \cite{Crainic:2004aa}).
\end{itemize}

\paragraph{Results in the simply connected case.} In the current paper we discuss only the solution of the lifting problem, leaving applications to index theory for the future. To this end, first one needs to specify the topological conditions under which the lifting problem is expected to admit a solution. In this first approach of ours, we pose the following ones:
\begin{enumerate}
\item The manifold $M$ is \emph{simply connected}. The latter ensures that the Hurewicz map is an isomorphism and this simplifies our constructions. The reason is that, the Crainic-Fernandes monodromy map factorizes through the Hurewicz map 
 (after tensoring $\pi_2(M)$ with $\R$ with the effect of doing away with possible torsion elements). More precisely, that is the Hurewicz map of the universal cover.
\item The group $\pi_2(M)$ is \emph{finitely generated}. Although this condition is rather mild, it is quite necessary: In \S \ref{sec:integrlifts} we show that otherwise the lifting problem has no solution.
\item The transitive algebroid $B$ is \emph{abelian}. In this case $B$ is classified by an element $[\omega] \in H^2_{dR}(M,V)$ where $V$ is the (abelian) isotropy Lie algebra. The integrability obstruction is exactly the discreteness of the group of periods of $\omega$. Apart from the computability of the integrability obstruction, working in the abelian case is justified because other, more general and interesting classes of algebroids can be written as abelian extensions (see examples \ref{exs:extns}). As the classification of abelian transitive Lie algebroids is an instance of the classification of extensions of abelian extensions of (not necessarily transitive) Lie algebroids (\cf \cite{MK2}), our constructions might also be valid in these more general cases.

\end{enumerate}
Under these assumptions, we show that the lifting problem can be answered as follows:
\begin{enumerate}
\item First we explain the Almeida-Molino counterexample: A representative $\omega$ can be decomposed as a sum of closed 2-forms with values in $V$ 
 $$\sum_{i=1}^n\varphi_i \otimes \Theta_i,$$ where $n$ is an integer number which depends on the rank of $\pi_2(M)$, the dimension of $V$ and the class of $\omega$. The form $\widehat{\omega} = (\varphi_1 \otimes \Theta_1,\ldots,\varphi_n \otimes \Theta_n)$ then defines the integrable lift $\widehat{B}$; the morphism $\widehat{B} \to B$ is just the map which sums the components.
 This construction depends on some choices but the integrating gropoids will be (in this Abelian case) Morita equivalent. For instance as far as one is concerned with $K$-Theory we expect that these choices will be irrelevant.
We call these constructions the \emph{Almeida-Molino} integrable lifts. When $B$ is integrable, we prove that $n=1$ and this construction gives nothing new.

\item There is a second solution to the lifting problem, which is universal. The construction is based on the observation that the familiar de Rham isomorphism on the manifold $M$ can be realized as a closed 2-form $\Theta$ with values in the real homology  group $H_2(M;\R)$. This form automatically has discrete periods, and the Lie algebroid it induces leads to a solution for the lifting problem of $B$.
\end{enumerate}

\paragraph{The non simply connected case.} When the base manifold $M$ is not simply connected, the Hurewicz map is not an isomorphism. In particular the lack of injectivity
 makes the implementation of the de Rham lifting construction quite difficult. The standard strategy to overcome these difficulties is to pull the given Lie algebroid $B \to M$ back to the universal cover via the covering projection $\pi : \widetilde{M} \to M$ and perform the de Rham lifting construction for $\pi^{\ast}B$. Then one has to somehow produce an integrable lift $\widehat{B}$ for $B$ from the integrable lift $\widehat{\pi^{\ast}B}$. 

To carry out this scheme, extra topological assumptions on the manifold $M$ are needed. For instance, this can be achieved when the fundamental group $\pi_1(M)$ is finite: The interested reader may check that the de Rham integrable lift $\widehat{\pi^* B}$ carries a natural $\pi_1(M)$ action. If $\pi_1(M)$ is finite, this action can be averaged out and we obtain an integrable $\widehat{B}$. 
However, imposing that $\pi_1(M)$ is finite is a rather strong topological condition. In order to give a more general solution to the lifting problem, in \S \ref{sec:nonsc} we impose a different topological assumption for $M$, which can probably be met in many more cases. This assumption is basically the existence of an equivariant family of forms which produces the second de Rham cohomology of $\widetilde{M}$. We show that this assumption amounts to the existence of a closed 2-form which is \textit{equivariant}, and whose associated Lie algebroid resolves the lifting problem for $\pi^{\ast}B$. In \S \ref{sec:solnonsc} we explain how this leads to a solution of the lifting problem for the original Lie algebroid $B$.

Finally, from private communication with Ioan Marcut, we are aware of a lifting construction of his. It solves the lifting problem in the simply connected case by adding the minimal possible number of dimensions.

\paragraph{Acknowledgements.} We kindly thank Georges Skandalis for proposing the problem and for a number of inspiring discussions. Also Ugo Bruzzo, Kirill Mackenzie, Ioan Marcut and Alan Weinstein for their useful insights. In particular,  I. Marcut provided us counterexample \ref{counter} and his insights on the monodromy map for transitive algebroids which are contained in section \ref{sec:CFobstr}. We also thank him for providing us notes of his concerning the lifting problem, at an early stage of this project. The idea of the ``de Rham'' construction is parallel to the ideas discussed in these notes.

This paper was partially supported by
the grant H2020-MSCA-RISE-2015-691246-QUANTUM DYNAMICS.
I. Androulidakis was partially supported by ANR-14-CE25-0012-0 and SCHI 525/12-1 (DFG) and FP7-PEOPLE-2011-CIG, Grant No PCI09-GA-2011-290823. P. Antonini was partially supported by ANR-14-CE25-0012-01, by SCHI 525/12-1 (DFG) and NWO Veni grant No 613.009.031.

\section{Abelian transitive Lie algebroids and their integrability}\label{sec:trans}

For the convenience of the reader, we recall in \S \ref{sec:transalgds} the classification of transitive Lie algebroids from \cite{MK2}, focusing on the abelian case in \S \ref{sec:abcentr}. In \S \ref{sec:CFobstr} we recall the integrability obstruction given by Crainic and Fernandes \cite{CFLie} for the transitive case. The abelian transitive Lie algebroids are the most favorable ones in view of the computability of this obstruction.  In \S \ref{sec:CFobstr} we recall the integrability problem and the associated obstruction and give some properties of this obstruction that we will need in this paper. The most important one is factorization through the  Hurewicz map. Last, in \S \ref{sec:integrlifts} we explain the problem of lifting the integrability obstruction of a transitive Lie algebroid.

Throughout the paper, $M$ will be a connected smooth manifold (finite dimensional). Recall that a  \textbf{transitive Lie algebroid} $B$ over $M$ is an extension of vector bundles
\begin{equation}\label{eqn:algd}
	\xymatrix{0\ar[r]&L\ar[r]^\iota&B\ar[r]^\rho &TM\ar[r]&0}
\end{equation}
together with a Lie bracket $[\cdot,\cdot]$ on the space of smooth section $\Gamma(B)$ satisfying the Leibniz rule:
\[[a,fb]=f[a,b]+\rho(a)(f)b, \ \ a,b\in \Gamma(B), \ f\in C^{\infty}(M).\]
The map $\rho$ is called the \textbf{anchor} of $B$; the induced map $\rho : \Gamma(B) \to \Gamma (TM)$ preserves the Lie brackets. The bundle $L=\ker(\rho)$ is a Lie algebra bundle, (next called LAB) and for $x\in M$, the Lie algebra $L_x$ is called the \textbf{isotropy Lie algebra} at $x$.

\subsection{Mackenzie's classification of transitive algebroids}\label{sec:transalgds}

A splitting $\sigma:TM \to B$ of \eqref{eqn:algd} determines a vector bundle isomorphism
\[TM\oplus L\simeq B,\ \ \  X\oplus \xi\mapsto \sigma(X)+\iota(\xi).\]
Pulling back the Lie algebroid structure through this isomorphism, the anchor map becomes the projection to $TM$ and the bracket
\begin{eqnarray}\label{bracket}
[X\oplus\xi,Y\oplus\eta] = [X,Y] \oplus \big(\nabla^{\sigma}_X\eta - \nabla^{\sigma}_Y\xi + [\xi,\eta] -R_{\sigma}(X,Y)\big),
\end{eqnarray} 
where $X,Y \in \Gamma(TM)$, $[X,Y]$ is the usual bracket of vector fields, and $\xi, \eta \in \Gamma(L)$; the expression $[\xi,\eta]$ stands for the pointwise bracket on $L$. The connection $\nabla^{\sigma}$ on $L$ is defined by $$\nabla^{\sigma}_X (\xi) = \iota^{-1}[\sigma(X),\iota(\xi)]$$ for every $X \in \Gamma(TM)$ and $\xi \in \Gamma (L)$. The 2-form $R_{\sigma} : TM \times TM \to L$ is the curvature of $\nabla^{\sigma}$.

Whence, given a transitive Lie algebroid $B$, the choice of a splitting $\sigma$ induces the \emph{normal form} $TM \oplus_{R_{\sigma}} L$ of $B$.

Moreover, we have the following results.
\begin{enumerate}
\item The operator $\nabla^{\sigma}$ is a connection on $L$ (in the sense of vector bundles) from the Leibniz rule, while, the Jacobi identity implies that $\nabla^{\sigma}$ is a Lie connection:
  $$\nabla^\sigma_X[\xi,\eta]=[\nabla^\sigma_X \xi,\eta]+[ \xi,\nabla^\sigma_X\eta], \quad \xi,\eta \in \Gamma(L).$$
\item The ``curvature form'' $R_\sigma \in  \Omega^2(M;L)$ is $$R_{\sigma}(X,Y) = \sigma[X,Y] - [\sigma(X),\sigma(Y)].$$

    Recall from \cite[Proposition 5.2.19]{MK2} that this 2-form is related with the curvature $R_{\nabla^{\sigma}}$ of the connection $\nabla^{\sigma}$ by the following identities:
\begin{eqnarray}\label{compatibility}
R_{\nabla^{\sigma}} = \operatorname{ad} \circ\, R_\sigma \quad\text{and}\quad \nabla^{\sigma}(R_\sigma)=0 \quad\text{(the Bianchi identity)},
\end{eqnarray}
where $\operatorname{ad}(a)(\xi)=[a,\xi]$, for $a\in \Gamma(B)$ and $\xi \in \Gamma(L)$.
We have extended $\nabla^\sigma$ to a covariant derivative $\nabla^\sigma: \Omega^n(M;L) \to \Omega^{n+1}(M;L)$ in the usual way.
\item Due to the Jacobi identity $\nabla^{\sigma}$ preserves the center $ZL$ and the restricted connection on $ZL$ does not depend on $\sigma$. We denote it $\nabla^{ZL}$. From \eqref{compatibility} we see that is flat.
\end{enumerate}
Conversely, 
\cite[Corollary 7.3.8]{MK2} says that: Given a Lie algebra bundle $(L,[\cdot,\cdot]) \to M$ with an $L$-valued 2-form $\omega\in \Omega^2(M;L)$ and a Lie connection $\nabla$ on $L$, which satisfies the compatibility equations \eqref{compatibility}, then formula  \eqref{bracket} determines a transitive Lie algebroid structure on $TM \oplus L$.

\subsection{Abelian transitive Lie algebroids}\label{sec:abcentr}

As we saw in \S \ref{sec:transalgds}, the favorable cases of transitive Lie algebroids are the ones with central curvature. Abelian transitive Lie algebroids are in this class, and here we describe them. 

\begin{definition}\label{dfn:centrextns}
\begin{enumerate}
A transitive Lie algebroid \eqref{eqn:algd} is called
\item \emph{Abelian}: if its isotropy is an abelian LAB. In other words, if it is an extension of vector bundles
\begin{equation}\label{eqn:abalgd}
0 \to E \stackrel{\iota}{\longrightarrow} B \stackrel{\rho}{\longrightarrow} TM \to 0
\end{equation}
such that $E$ carries the zero Lie bracket.
\item \emph{Trivially abelian}: it is abelian and the bundle $E = M \times \R^{\ell}$ in \eqref{eqn:abalgd} is trivial as a flat bundle. Namely there exists a splitting $\sigma : TM \to B$ of the exact sequence
\begin{equation}\label{eqn:centralgd}
0 \to M \times \R^{\ell} \stackrel{\iota}{\longrightarrow} B \to TM \stackrel{\pi}{\longrightarrow} 0 
\end{equation}
such that the connection $\nabla^{\sigma}$ it induces in $M \times \R^{\ell}$ is the standard flat connection $\nabla^0_X(f)=L_X(f)$ for every $X \in \Gamma (TM)$ and $f \in C^{\infty}(M,\R^{\ell})$.
\end{enumerate}
\end{definition}

As we explained in \S \ref{sec:transalgds}, for any splitting $\sigma$ of an abelian Lie algebroid \eqref{eqn:abalgd}, the connection $\nabla^{\sigma}$ of $E$ is flat. Put $V$ the fiber of $E$. It follows from the Ambrose-Singer theorem that the holonomy map $h^{\sigma} : \pi_1(M) \to GL(V)$ has discrete image $h^{\sigma}(\pi_1(M))$. But the group $h^{\sigma}(\pi_1(M))$ may not be trivial, even if the bundle $E$ is a product $M \times V$.  On the other hand, a trivially abelian transitive Lie algebroid \eqref{eqn:centralgd} has a splitting $\sigma$ such that $h^{\sigma}(\pi_1(M))$ is trivial. 

\begin{prop}\label{prop:trivab}
Let $B$ be an abelian transitive Lie algebroid \eqref{eqn:abalgd}. If its base manifold $M$ is simply connected, then $B$ is (isomorphic to) a trivially abelian transitive Lie algebroid.
\end{prop}
\begin{proof}
Since $M$ is simply connected, the flat bundle $E$ can be trivialized using parallel translation, so we may consider a priori that $E=M \times \R^{\ell}$. Let $\sigma$ be a splitting. The connections $\nabla^{\sigma}$ and $\nabla^0$ are related by 
\begin{equation}\label{connections}
\nabla^{\sigma}_X(\mu)=L_X(\mu)+\phi(X)(\mu)
\end{equation}
for some 1-form $\phi : TM \to \operatorname{End}(E)=M \times \mathfrak{gl}(V)$. 

\paragraph{Claim:} The map $\phi$ satisfies the Maurer-Cartan equation $d\phi + [\phi,\phi]=0$.

Let $X, Y \in \Gamma(TM)$ and a section $\mu : M \to \mathfrak{gl}(V)$ of $\operatorname{End}(E)$. From 
\eqref{connections} we find:
\begin{equation}\label{curvatures}
\begin{split}
R_{\nabla^{\sigma}}(X,Y)(\mu) &= R_{\nabla^0}(X,Y)(\mu) + (\phi(X)(L_Y(\mu))-L_Y(\phi(X)(\mu))) \\ &- (\phi(Y)(L_X(\mu)-L_X(\phi(Y)(\mu))) + [\phi(X),\phi(Y)](\mu) - \phi([X,Y])(\mu)
\end{split}
\end{equation}
From the chain rule we have $L_X(\phi(Y))(\mu) = L_X(\phi(Y)(\mu))-\phi(Y)(L_X(\mu))$. Plugging this into \eqref{curvatures} we conclude.

The Maurer-Cartan equation says that $\phi$ is a morphism of Lie algebroids, whence it integrates to a morphism of Lie groupoids $F : \Pi(M) \to GL(V)$, where $\Pi(M)$ is the fundamental groupoid of $M$. Since $\pi_1(M)=0$, the map $(s,t) : \Pi(M) \to M \times M$ is an isomorphism, whence the map $F$ can be seen as an isomorphism of vector bundles $F : M \times V \to M \times V$. Now consider the connection $\nabla_X(F(\mu)) = F(L_X(\mu))$. Whence, composing $F$ with the inclusion map $M \times V \to B$ we reformulate $B$ so that the connection $\nabla^{\sigma}$ becomes the standard flat connection $\nabla^0$.
\end{proof}

\begin{cor}\label{cor:trivab}
Let $B$ be an abelian transitive Lie algebroid over $M$ and $\pi : \widetilde{M} \to M$ the universal cover of $M$. Then the pullback algebroid $\pi^{!}B$ is trivially abelian.
\end{cor}

\begin{ex}\label{exs:extns}
Consider a closed 2-form $\omega$ on $M$. Then $B_{\omega} = TM \oplus_{\omega} (M \times \R)$ is a (transitive) Lie algebroid with anchor map the projection to $TM$ and Lie bracket 
\begin{equation}\label{brack} 
[X \oplus V,Y\oplus W] = [X,Y] \oplus \{X(W)-Y(V)-\omega(X,Y)\}.
\end{equation}
 We obtain the abelian extension $0 \to M \times \R \to B \to TM \to 0$. 
 This algebroid is strongly related with geometric prequantization : Mackenzie in \cite[Thm 8.3.9]{MK2} showed, in the simply connected case, that  $B_{\omega}$ is \emph{integrable} if and only if the de Rham cohomology class of $\omega$ belongs to the image, inside $H^2(M;\R)$, of $\check{H}^2(M;D)$ for a discrete subgroup $D \subset \R$. Since such a subgroup is always in the form $a\Z$ for its minimum positive number $a$, then, up to scale we're saying that $\omega$ has to be an \emph{integral} form. Crainic explained the precise relation between \emph{integrality} and \emph{integrability} in 
 \cite{Crainic:2004aa}.
 
  If $M$ is simply connected, then $B_{\omega}$ is trivially abelian. Otherwise, consider the universal cover $\pi : \widetilde{M} \to M$. The pullback of $B_{\omega}$ to $\widetilde{M}$ is the transitive algebroid $B_{\pi^*\omega} = T\widetilde{M}\oplus_{\pi^*\omega}(\widetilde{M} \times \R)$ and it is trivially abelian.
\end{ex}

\begin{exs}
Recall from \cite{MK2} that the classification given in \S \ref{sec:transalgds} holds for Lie algebroid extensions beyond the transitive case. In view of this, we give two examples of abelian extensions as such. We will not be concerned with these cases here, but the results we give in this paper might be generalized to them.
\begin{enumerate}
\item Given a transitive Lie algebroid \eqref{eqn:algd}, consider the adjoint representation $\operatorname{ad} : B \to \mathcal{D}[L]$ (see \cite[Ex. 8.3.3]{MK2}). Its image $\operatorname{ad}[L]$ can be identified with the quotient $B/ZL$. Moreover, the Lie algebroid $\mathcal{D}[L]$ of covariant differential operators integrates to the frame groupoid of $L$, and the subalgebroid $\operatorname{ad}[L]$ integrates to the groupoid of inner automorphisms between fibers of $L$. Whence, any transitive Lie algebroid can be written as an abelian extension 
\begin{equation}\label{eqn:abexint}
0 \to ZL \to B \to A(G) \to 0 
\end{equation}
where $G$ is a transitive Lie groupoid. 

\item It was shown in \cite{AZ3} that every almost regular Poisson structure $(M,\Pi)$ is an abelian extension of Lie algebroids 
\begin{equation}\label{eqn:almregextn}
0 \to E \to T^{\ast}M \stackrel{\Pi}{\longrightarrow} A(H(\cF)) \to 0
\end{equation}
where $H(\cF)$ is the holonomy groupoid of the associated symplectic foliation on $M$.
\end{enumerate}
\end{exs}



\subsection{The monodromy map}\label{sec:CFobstr}

We recall here the integrability process for transitive Lie algebroids given by Crainic and Fernandes \cite{CFLie} and (independently) Iglesias \cite{Iglesias:1995aa}. We also give properties that will be needed later.

First recall that a Lie algebroid $B$ over $M$ is integrable if there exists a Lie groupoid $\cG(B) \gpd M$ whose Lie algebroid $A\cG(B)$ is $B$.

Consider a transitive Lie algebroid $0 \to L \to B \to TM \to 0$. Let $m\in M$ and $G(L_m)$ be a simply connected Lie group with Lie algebra $L_m$. Put $Z(G(L_m))$ be the center of $G(L_m)$. It was shown in \cite{CFLie} that the obstruction to the integrability of $B$ is controlled by a certain group homomorphism:
\[\partial_m : \pi_2(M,m) \lra Z(G(L_m)),\]
called the \emph{monodromy map}. Namely, $B$ is integrable if its image i.e, the \emph{monodromy group}
\[\widetilde{N}_m:=\partial_m(\pi_2(M,m))\subset  Z(G(L_m))\]
is a discrete subgroup of $Z(G(L_m))$. More specifically, given $B$ one can always construct a topological groupoid $\cC \gpd M$ with simply connected $s$-fibers. The monodromy group is discrete if and only if $\cC$ is a Lie groupoid and its Lie algebroid is $B$. 

The construction of the monodromy map mimics the construction of the boundary map in homotopy: a smooth sphere in $M$ is lifted to a Lie algebroid morphism from the disk to $B$; the restriction to the boundary gives a curve in $L_m$; using right translation, this curve integrates to a unique curve in $G(L_m)$ starting at $e$; the monodromy map returns the end point of this curve (see \cite{CFLie} for details).

Let us note some important properties of the monodromy group:
\begin{enumerate}
\item Let $Z(G(L_m))^{\circ}$ be the connected component of the identity for the group $Z(G(L_m))$. The restriction of the exponential map provides a group isomorphism $\operatorname{exp}_{|} : Z(L_m) \to Z(G(L_m))^{\circ}$. Whence $\operatorname{exp}^{-1}$ identifies the intersection  $\widetilde{N}_m \cap Z(G(L_m))^{\circ}$ with a subgroup $N_m$ of $Z(L_m)$. Obviously, $N_m$ is discrete if and only if $\widetilde{N}_m$ is discrete. This justifies why, abusing the terminology we call $N_m$ the monodromy group. Also by abuse of notation we will call monodromy map the composition $(\operatorname{exp}_{|}^{-1}\circ \partial_m) : \pi_2(M,m) \to N_m$. We continue to denote this map $\partial_m$ as well.
\item The monodromy map $\partial_m : \pi_2(M,m) \to ZL_m$ is a group morphism from an abelian group to a vector space. It follows that $\partial_m([\gamma])=0$ for every torsion element $[\gamma] $ of $\pi_2(M,m)$.
\item As noted in \cite{CFLie}, the monodromy map can be computed explicitly when the algebroid $B$ is abelian, namely when $ZL=L$. Considering a splitting $\sigma:TM \to B$ with associated curvature $\omega \in \Omega^2(M,L)$, the monodromy map becomes:
\begin{equation}\label{eqn:boundary}
\partial_m:\pi_2(M,m) \lra G(L_m), \quad
[\gamma]\longmapsto \exp\big(\int_{\mathbb{S}^2}\gamma^*(\omega)\big).
\end{equation}
The integral is defined for a smooth representative $\gamma:\mathbb{S}^2 \lra M$ with $\gamma(\textrm{north pole})=m$ as follows: by pulling back $L$ to the sphere we get a flat bundle on $\mathbb{S}^2$ that we trivialize by parallel transport: $\gamma^*(L)\cong \mathbb{S}^2 \times L_m$. Then $\gamma^*(\omega)\in \Omega^2(\mathbb{S}^2)\otimes L_m$ can be integrated over $\mathbb{S}^2$ to an element of $L_m$. Alternatively, one can pull back $L$ to the universal cover $p:\widetilde{M}\to M$ where it can be trivialized by parallel transport, $p^*(L)\simeq \widetilde{M}\times L_m$, and then $p^*(\omega)\in \Omega^2(\widetilde{M})\otimes L_m$ can be integrated canonically over $\pi_2(\widetilde{M},\tilde{m})$. On the other hand, the projection $p$ induces an isomorphism $\pi_2(\widetilde{M},\tilde{m})\simeq\pi_2(M,m)$, for $\tilde{m}\in p^{-1}(m)$. 
\item When $B$ is also trivially abelian, one can prove immediately that the monodromy map $\partial_m$ factors trough the Hurewicz map $\pi_2(M,m) \longrightarrow H_2(M)$. Let $\omega \in \Omega^2(M;\R^\ell)$ be the curvature form. We have the following commutative diagram
$$\xymatrix{\pi_2(M,m)\ar[d]^{\iota}\ar[rr]^\partial & {} &\R^{\ell}\\
\pi_2(M,m)\otimes_{\mathbb{Z}}\mathbb{R}\ar[rr]_{H \otimes \operatorname{Id}}&{}&H_2(M;\R)\ar[u]_{\int\omega}}$$
in which the map $\iota$ is $[\gamma] \mapsto \gamma \otimes_{\Z}1$; the map $H : \pi_2(M) \to H_2(M)$ is the Hurewicz morphism and the $\int \omega$ is the integration map.

\item An important property of the monodromy map is that it behaves \textbf{functorially}. Namely, for any morphism $f : B\to B'$ of transitive algebroids which covers the identity, the induced map between the simply connected Lie groups ${F_m:G(L_m)\to G(L_m')}$ intertwines the monodromy maps:
\[
\xymatrix{
\pi_2(M,m) \ar[d]_{\partial^{B}_m} \ar[dr]^{\partial^{B'}_{m}} &\\
Z(G(L_m)) \ar[r]^{F_m}& Z(G(L'_m))}
\]
Denoting by $N'_m$ the monodromy group of $B'$ at $m$, the above implies that
\begin{equation}\label{eq:monodromy_groups}
F_m(N_m)=N'_m.
\end{equation}
\end{enumerate}


\subsection{Integrable lifts and the prerequisite for their existence}\label{sec:integrlifts}

The main problem studied in this paper is the following: given a transitive Lie algebroid $B$ over $M$, construct a integrable transitive Lie algebroid $\widehat{B}$ over $M$ and a surjective morphism:
\[0\lra \ker(f)\lra \widehat{B\,}\stackrel{f}{\lra} B\lra 0.\]
Such an extension will be called an \textbf{integrable lift} of $B$. We will present two different constructions of such integrable lifts in the case $M$ is simply connected and $B$ is abelian. In principle, we are interested in constructions which are canonical and have the minimal dimension possible. 

First, let us remark that integrable lifts do not always exist:
\begin{ex}\label{counter}
Let $M$ be a smooth manifold, and let $m\in M$. Consider the group homomorphism \[\mathrm{H}_{\R}:\pi_2(M,m)\lra H_2(M,\mathbb{R})\]
obtained as the composition of the Hurewicz map $\mathrm{H}:\pi_2(M,m)\to H_2(M,\mathbb{Z})$ and of the canonical map $H_2(M,\mathbb{Z})\to H_2(M,\mathbb{R})$. Assume that $\mathrm{H}_{\R}(\pi_2(M,m))$ spans an infinite dimensional subspace of $H_2(M,\mathbb{R})$. Thus, there is a sequence in $\{\gamma_k\}_{k\geq 1}$ in $\pi_2(M,m)$ such that the elements $\{\mathrm{H}_{\R}(\gamma_k)\}_{k \geq 1}$ are linearly independent in $H_2(M,\mathbb{R})$. Concretely, one can take $M=\R^3\backslash \mathbb{Z}^3$. Fix numbers $a_{k}\in \mathbb{R}$ which are linearly independent over $\mathbb{Q}$. By the De Rham Theorem we have
\[H^2_{\mathrm{dR}}(M,\mathbb{R})\cong \operatorname{Hom}(H_2(M,\mathbb{R}),\mathbb{R})\]
Whence there exists a closed 2-form $\omega\in \Omega^2(M)$ such that
\[\int_{\mathbb{S}^2}\gamma_k^*(\omega)=a_k, \ \ \textrm{for all}\ \ k\geq 1.\]
To the closed 2-form $\omega$, we associate the algebroid $B_\omega$ in normal form with bracket \eqref{brack} and monodromy group:
 \[N_m=\Big\{\int_{\mathbb{S}^2}\gamma^*(\omega)\ :\ [\gamma]\in \pi_2(M,m)\Big\}\subset \R.\]
We claim that there is no integrable lift of $B_\omega$. On the contrary, assume that such a lift $f : \widehat{B\,}\to B_{\omega}$ existed, and denote its isotropy Lie algebra $L_m$. Then $f$ restricts to a Lie algebra map $f_m : L_m\to \R$, which integrates to a group homomorphism between the simply connected Lie groups: $F_m:G(L_m)\to \R$. Let $\widehat{{N}}_m$ be the monodromy group of $\widehat{B\,}$. Our assumption is that $\widehat{N}_m$ is discrete, and so $G(L_m)/\widehat{N}_m$ is a connected Lie group. Since any connected Lie group is homotopy equivalent to its maximal compact Lie group \cite{Malcev45}, compact manifolds have finitely generated fundamental group, and $\pi_1(G(L_m)/N_m,e)\simeq \widehat{N}_m$, it follows that $\widehat{N}_m$ is a finitely generated abelian group. Thus also $N_m=F_m(\widehat{N}_m)$ is finitely generated, which is impossible because $N_m$ contains the sequence $\{a_k\}_{k\geq 1}$ of elements which are linearly independent over $\mathbb{Q}$.
\end{ex}

The arguments in the example above can be easily generalized to any abelian transitive Lie algebroid:
\begin{prop}\label{prop:fingen}
Let $B$ be an abelian transitive Lie algebroid. If the monodromy group of $B$ is not finitely generated, then $B$ does not admit an integrable lift.
\end{prop}
\begin{remark}\label{rmk:fingen}
The monodromy group of $B$ is the image of the monodromy map, whose domain is $\pi_2(M)$. Due to Prop. \ref{prop:fingen}, the lifting problem is allowed to have a solution under the (rather mild) topological assumption that the group $\pi_2(M)$ has finite  rank.
\end{remark}

\section{The Almeida-Molino integrable lift}\label{sec:AlmMol}

We give here a solution to the lifting problem which is inspired by the Almeida-Molino non-integrable algebroids \cite{AlmeidaMolino}, as well as the factorization of the monodromy map with the Hurewicz morphism. In fact, our construction is really an explanation of the Almeida-Molino algebroids.

The ``Almeida-Molino'' integrable lift can be constructed under the following assumptions for the transitive Lie algebroid $B \to M$:

\begin{assumptions}\label{assumption2}
\begin{enumerate}
	\item The group $\pi_2(M)$ has finite rank. 
	\item The Hurewicz map with real coefficients $\pi_2(M)\otimes \R  \longrightarrow H_2(M)\otimes_{\Z}\R$ is an isomorphism.
	\item $B$ is a trivially abelian transitive Lie algebroid $\xymatrix{0\ar[r]&M \times \R^{\ell}\ar[r]&B\ar[r]&TM\ar[r]&0}$.
\end{enumerate}	
\end{assumptions}
For instance, the second and third items in assumptions \ref{assumption2} are satisfied if $M$ is simply connected. As for the first assumption, as mentioned in Remark \ref{rmk:fingen}, it ensures that a solution of the lifting problem of $B$ is allowed.

Without further ado, we adopt assumptions \ref{assumption2} throughout the current section.

\subsection{Dimension and the second homotopy group}\label{sec:irrdeg}

An interesting feature of this constructive approach to the lifting problem, is that it uses a canonical dimension induced by the group $\pi_2(M)$.

 Let $r=\operatorname{rank} \pi_2(M)$ and put $\tau \pi_2(M)$ the torsion subgroup of $\pi_2(M)$. The number $r$ is the cardinality of any maximally independent system (over $\mathbb{Z}$), say $\{[\gamma_1],...,[\gamma_r]\} \subset \pi_2(M,m)$ such that the images of the $\gamma_i$ in $\pi_2(M)/\tau\pi_2(M)$ generate $\pi_2(M)/\tau\pi_2(M)$. Since $M$ is a manifold we can assume that every $\gamma_i: \mathbb{S}^2 \to M$ is a smooth map. We call $\Gamma=(\gamma_1,...,\gamma_r)$ this choice of a system of generators. 

From the assumption on the Hurewicz map, and by the universal coefficients theorem, we get an isomorphism of vector spaces $$H \otimes \operatorname{Id}_{\mathbb{R}^{\ell}} : \pi_2(M)\otimes_{\Z}\R^{\ell} \longrightarrow H_2(M;\Z)\otimes_{\Z}\R^{\ell}\cong H_2(M;\R^\ell).$$
Let $e_j$, $j=1,...,\ell$ be the canonical basis of $\mathbb{R}^{\ell}$. Since the $\gamma_i$ are a maximal independent set, we get the following basis of the homology:
$$\mu_{ij}:= H\otimes \operatorname{Id}_{\mathbb{R}^{\ell}}([\gamma_i]\otimes e_j) \in H_2(M;\mathbb{R}^{\ell}),\quad i=1,...,r, \, j=1,...,\ell$$
This defines the first arrow in the following chain of isomorphisms:
\begin{equation}\label{chain}
\pi_2(M)\otimes_{\mathbb{Z}}\mathbb{R}^{\ell} \longrightarrow H_2(M,\mathbb{R}^{\ell}) \longrightarrow H_2(M;\mathbb{R}^{\ell})^* \longrightarrow H^2_{dR}(M;\mathbb{R}^{\ell}).
\end{equation}
The second arrow here is the identification of a vector space with its dual, defined in terms of the basis $\mu_{ij}$. The last one is the inverse of the de Rham isomorphism $\operatorname{dR}:H^2_{dR}(M;\mathbb{R}^{\ell}) \longrightarrow H_2(M;\mathbb{R}^{\ell})^* $. Following the chain we end up with a corresponding basis of the de Rham cohomology:
$$[\varphi_i \otimes e_j]\in H^2_{\operatorname{dR}}(M;\mathbb{R}^{\ell}), \quad i=1,...,r, \, j=1,...,\ell.$$
This basis is defined by closed forms $\varphi_i \in \Omega^2(M;\mathbb{R})$. Its simple tensors form follows by the definition of the employed maps. Also the coefficients $\varphi_i$ are normalised by the de Rham theorem:
\begin{equation}\label{normalization}
\int_{\gamma_{j}}\varphi_i = \delta_{ij}, \quad i,j=1,...,r.
\end{equation}
Finally, choose a splitting $\sigma : TM \to B$ and consider its associated curvature $\omega$. Recall that the algebroid $B$ is isomorphic to $B_{\omega} = TM \oplus_{\omega}(M \times \R^{\ell})$ and also that 
the Bianchi identity is equivalent to the fact that $\omega$ is a closed 2-form in $\Omega^2(M,\R^{\ell})$

Using the basis $[\varphi_i \otimes e_j]$, the de Rham class of $\omega$ is written as:
\begin{equation}\label{omegainbasis}
[\omega]=\sum_{i=1}^r \sum_{ j=1}^{\ell} \alpha_{ij}[\varphi_i \otimes e_j], \quad \alpha_{ij} \in \mathbb{R}.  
\end{equation}
\begin{remark}
We can assume that the identity \eqref{omegainbasis} holds true also at the level of forms. Otherwise $\omega=\sum_{i=1}^r \sum_{ j=1}^{\ell} \alpha_{ij}\varphi_i \otimes e_j -d\rho$ for some $\rho \in \Omega^1(M;\mathbb{R}^{\ell})$ and we could work with the algebroid $B_{\omega'}$ where $\omega'=\omega + d\rho$. This is isomorphic to $B_\omega$ under the isomorphism $X \oplus V \mapsto X \oplus (V+ \rho(X))$. \\
\end{remark}
Now we consider the number $n\big{(}\Gamma\big{)}$ of non zero rows of the matrix $[\alpha_{ij}]\in M_{r \times {\ell}}(\mathbb{R})$. Clearly it depends on the choice $\Gamma.$
 If we define the vectors $\Theta_i:= \sum_{j=1}^{\ell} \alpha_{ij}e_j \in \mathbb{R}^{\ell}$ then we can write 
$$\omega= \sum_{i=1}^r \varphi_{i} \otimes \Theta_i.$$ Whence $n(\Gamma)$ is the number of non-zero vectors $\Theta_i$.

Assuming the cohomology class of $\omega$ is non-trivial we have $1\leq n(\Gamma)\leq r$. It turns out that the dimension $n(\Gamma)$ is related with the integrability of $B_{\omega}$.

\begin{prop}\label{prop:irrdeg}
Put $n(B_{\omega}):=\operatorname{min}\{n(\Gamma)\}$ when we vary all the possible choices $(\Gamma)$. If $n(B_{\omega})=1$ then $B_{\omega}$ is integrable. \end{prop}
\begin{proof}
Let $(\Gamma)$ be a choice realizing $n(\Gamma)=n(B_{\omega})=1$. This means that in the corresponding matrix $[\alpha_{i,j}]$ only one row is non-zero.
Let's assume that it is the first row. Therefore $\omega = \varphi_1 \otimes \Theta_1$ and the image of the monodromy map is the image of the map 
$\mathbb{Z}^k \ni (x_1,...,x_r) \longmapsto \sum_{\ell=1}^r x_{\ell}\int_{\gamma_{\ell}}\varphi_1 \otimes \Theta_1 = x_1 \Theta_1$ which is discrete. 
\end{proof}

\subsection{Construction of the Almeida-Molino integrable lift}

Now we can give the explicit construction of an integrable lift for $B_{\omega}$. This will be done by adding as many copies of the isotropy Lie algebra, as the dimension $n(B_{\omega})$ discussed in Prop. \ref{prop:irrdeg}.

\begin{thm}\label{simplyconnected}
Let $M$ be a manifold whose $\pi_2(M)$ has finite rank $r$ and such that the Hurewicz map with real coefficients $\pi_2(M)\otimes \R  \longrightarrow H_2(M)\otimes_{\Z}\R$ is an isomorphism.
Let  $$\xymatrix{0\ar[r]&M \times \R^\ell \ar[r]  &B\ar[r]&TM\ar[r]&0}$$ be a trivially abelian transitive algebroid. Assume $n(B) \geq 2$. Then,
 for every choice $(\Gamma)$ realizing $n(\Gamma)=n(B)$ we can construct an extension of transitive Lie algebroids
$$\xymatrix{0\ar[r]&K\ar[r]&B_{\overline{\omega}}\ar[r]^\mu& B_{\omega} \ar[r]&0}$$
 where 
 $B_{\overline{\omega}}$ 
 is integrable. In particular 
$B_{\overline{\omega}}=TM\oplus_{\overline{\omega}} {\R}^{n(B) \times \ell}$, for a closed 2-form $\overline{\omega} \in \Omega^2(M, {\R}^{n(B) \times \ell)})$ which is determined by $(\Gamma)$. 
 \end{thm}
 
\begin{proof}
\begin{enumerate}
\item Put ${\bf{n}}:=n(B)$. Let $\Gamma=([\gamma_1],...,[\gamma_r]) \in \pi_2(M)^r$ be any such independent set. By the previous discussion we have 
$\omega= \sum_{i=1}^r \varphi_{i} \otimes \Theta_i$. We are assuming that only ${\bf{n}}$ of the vectors $\Theta_i$ are non zero; we can arrange for them to be exactly $\Theta_1,...,\Theta_{{\bf{n}}}$.  
Now let's define
$$\overline{\omega}:=(\varphi_1 \otimes \Theta_1,...,\varphi_{{\bf{n}}} \otimes \Theta_{{\bf{n}}}) \in \Omega^2(M;(\mathbb{R}^\ell)^{{\bf{n}}}) = \Omega^2(M;\R^{\bf{n} \times \ell}),$$ and consider:
\begin{enumerate}
\item The algebroid $B_{\overline{\omega}}:=TM \oplus_{\overline{\omega}} \R^{\bf{n}\times \ell}$.
\item The map $\mu : B_{\overline{\omega}} \to B_{\omega}$ is the addition $(X, \xi_1,...,\xi_{{\bf{n}}}) \mapsto (X, \xi_1+ \cdots +\xi_{{\bf{n}}}).$ It is clearly a morphism of algebroids.
 \end{enumerate}
 
\paragraph{Claim:} The algebroid $B_{\overline{\omega}}$ is integrable. 

To prove this we have to show that the monodromy group $N_m(B_{\overline{\omega}})$ is discrete. With respect to the system $\Gamma$ of $\pi_2(M)$, the monodromy map $\partial:\pi_2(M) \to (\mathbb{R}^{\ell} )^{{\bf{n}}}$ is given by the formula:
$$\mathbb{Z}^r \ni (m_1,...,m_r) \longmapsto \Bigg{(} \int_{\sum_{i=1}^r m_i [\gamma_i]}\varphi_1 \otimes \Theta_1,..., \int_{\sum_{i=1}^r m_i [\gamma_i]}\varphi_{{\bf{n}}} \otimes \Theta_{{\bf{n}}}  \Bigg{)}$$
(Recall from \S \ref{sec:CFobstr} (item (b)) that the boundary map $\partial$ sends the torsion elements to zero.)

Let us look at its $\underline{i}$-th component: Due to the normalization \eqref{normalization}, this is the function  
$$(m_1,...,m_r) \longmapsto \int_{\sum_{i=1}^r m_i [\gamma_i]} \varphi_{\underline{i}} \otimes \Theta_{\underline{i}}=  m_{\underline{i}} \sum_{j=1}^k \alpha_{\underline{i}j}\varphi_{\underline{i}}\otimes e_j= m_{\underline{i}}\Theta_{\underline{i}}$$
The image of $\partial$ is discrete:
for $(m_1,...,m_r)\neq (m_1',...,m_r')$ we estimate the distance $\|\partial(m_1,...,m_r)-\partial(m_1',...,m_r')\|$. 
Then 
\begin{equation*}
\begin{cases}
\|\partial(m_1,...,m_r)-\partial(m_1',...,m_r')\|=0, \quad \textrm{if }(m_1,...,m_{{\bf{n}}})=(m_1',...,m_{{\bf{n}}}')\\
\|\partial(m_1,...,m_r)-\partial(m_1',...,m_r')\| \geq \min\{ \|\Theta_1\|,...,\|\Theta_{{\bf{n}}}  \|\}>0, \quad \textrm{if }(m_1,...,m_{{\bf{n}}})\neq(m_1',...,m_{{\bf{n}}}').
\end{cases}
\end{equation*}
\end{enumerate}
\end{proof}

\begin{terminology}
Let $M$ be a smooth manifold such that the degree-2 Hurewicz map is an isomorphism (\eg $M$ is simply connected). Let $B$ a trivially abelian transitive Lie algebroid over $M$. The integrable lift constructed in theorem \ref{simplyconnected} is called the \emph{Almeida-Molino integrable lift} of $B$.
\end{terminology}

\begin{remarks}\label{rmk:nonab}
\begin{enumerate}
\item A priori different choices $(\Gamma)$ give different (non isomorphic) extensions. However it is straightforward to see that the integrating groupoids corresponding to different choices are Morita equivalent. This is particularly useful when one deals with the $K$-theory of the associated $C^*$-algebras.
\item This method works also if we start with  a simply connected $M$, and a transitive algebroid $\xymatrix{0\ar[r]&L \ar[r]  &B\ar[r]&TM\ar[r]&0}$ with non necessarily abelian $L$ admitting a splitting whose curvature lies in $\Omega(M,ZL)$. Indeed, after having trivialized $ZL$ by parallel translation, to make the sum map $\mu$ a morphism, we could construct the corresponding $\overline{\omega}$ to take values in $ZL^{\rtimes n(B)}$ which is a semidirect extension with respect to the adjoint action.
\end{enumerate}
\end{remarks}

\begin{ex}\label{ex:AlmMol}
The Almeida-Molino counterexample to integrability \cite{AlmeidaMolino} is the Lie algebroid $B_{\lambda} = TM \oplus_{\omega}(M \times \R)$, where $(M,\omega)$ is the symplectic manifold $M=S^2 \times S^2$, with $\omega = \omega_1 + \lambda\omega_2$, where $\lambda$ is irrational and $\omega_i = pr_i^{\ast}(V)$, for the standard volume form $V$ of $S^2$. The monodromy group is exactly the group of periods of $\omega$, which we calculate to be $\Z + \lambda\Z$. It is a dense subgroup of $\R$, whence $B_{\lambda}$ is not integrable. 

Now, the dimension mentioned in Prop. \ref{prop:irrdeg} is $n(B_{\lambda})=2$ and the Almeida-Molino integrable lift is $\overline{B_{\lambda}} = TM \oplus_{\overline{\omega}} (M \times \R^2)$, where the form $\overline{\omega} \in \Omega^2(M,\R^2)$ is $\overline{\omega} = (\omega_1,\lambda\omega_2)$. We obtain the short exact sequence $$0 \to M \times \R \stackrel{\iota}{\longrightarrow} \overline{B_{\lambda}} \stackrel{\mu}{\longrightarrow} B_{\lambda} \to 0$$ where $\mu$ is the addition map and $\iota$ is the map $t \mapsto (t,\lambda t)$ for all $t \in \R$.

Following \cite{Iglesias:1995aa} and \cite{Crainic:2004aa} we find that the isotropy group of the monodromy groupoid $G(\overline{B_{\lambda}})$ is $T^2 = \R^2 / Per(\overline{\omega})$, namely  $G(\overline{B_{\lambda}})$ corresponds to a torus bundle. Now, passing to the monodromy groupoids of the above extension, we obtain the following short exact sequence of topological groupoids: $$1 \to M \times \R \stackrel{\iota}{\longrightarrow} G(\overline{B_{\lambda}}) \stackrel{\mu}{\longrightarrow} G(B_{\lambda}) \to 1$$ (where the maps are abusively also denoted $\iota$ and $\mu$). Note that the map $\iota : \R \to T^2$ is  the immersion of the real line in the torus by irrational rotation. In ultimate analysis, the non integrability phenomenon here amounts by a \emph{bad quotient} of a Lie group by a non closed subgroup; a situation which is studied, and often encountered in non commutative geometry
Whence, the non-smooth transitive groupoid $G(B_{\lambda})$ is the quotient $G(\overline{B_{\lambda}}) / \iota(M \times \R)$ and the isotropy group $G(B_{\lambda})_x^x$ is closely related with the leaf space of the irrational rotation foliation.
\end{ex}

\begin{remark}
The leaf space viewpoint  of the Almeida-Molino counterexample that we discussed in Example \ref{ex:AlmMol}, was initiated by the following observation, regarding the integrability obstruction given by Crainic and Fernandes in \cite{CFLie}. When we started to investigate the Molino-Almeida counterexample we noticed the following very interesting  coincidence! Consider the irrational rotation foliation on the torus $T^2$. Its leaves are the integral curves of the vector field $X = \partial_x + \lambda\partial_y$ (on $\R^2$). Now Let  $\Delta = -X^2$ be the longitudinal Laplace operator. Applying Fourier transform we find that the spectrum of this operator coincides with the monodromy group $\Z + \lambda\Z$. This observation was the first indication that an irrational rotation foliation might be related with the non integrability phenomenon here.

Actually, this observation extends beyond the integrability obstruction: Recall \cite{FackSkandalis} that when a foliation has a dense leaf, then the closure of the spectrum of  the longitudinal Laplacian coincides with the (closure of) the spectrum of the Laplacian of this leaf. In the case of the irrational rotation, the latter is the Laplacian $-(\partial_x)^2$, whose spectrum is $[0,+\infty]$, obviously coinciding with the closure of $\Z + \lambda\Z$. This reflects the fact that $G(B_{\lambda})$ is the quotient $G(\overline{B_{\lambda}}) / \iota(M \times \R)$.

It is quite tempting to hope that having a solution of the lifting problem for transitive and abelian Lie algebroids implies that the above observations regarding the Almeida-Molino example, are true for a much larger class of algebroids. In a future paper we intend to explore this.
\end{remark}


\section{The de Rham integrable lift (simply connected case)}\label{sec:universal}

In this section we give a second solution to the lifting problem for abelian transitive Lie algebroids over a \textbf{simply connected} manifold $M$. 
Specifically, we intepret the de Rham theorem on a manifold as giving a transitive Lie algebroid over $M$ which is always integrable. This construction is canonical and, in a sense, universal. 
The price to pay is that the dimension of the integrable lift we construct may be higher as compared to the previous Almeida-Molino lift. This is explained in \S \ref{sec:max}.

We will use the de Rham isomorphism to associate an isomorphism class of integrable Lie algebroids to every simply connected manifold. Our assumptions will be:
\begin{assumptions}\label{asm}
We consider trivially abelian Lie algebroids $B \to M$ such that $M$ is simply connected and 
the group $\pi_2(M)$ is finitely generated.
\end{assumptions}
Now let us gather here some observations that will be useful in this section:
\begin{enumerate}
\item The universal coefficients theorem identifies $H_2(M)\otimes_{\Z}\R$ with the vector space $H_2(M,\R)$.
 The map $\iota_{\R}:H_2(M) \longrightarrow H_2(M)\otimes_{\Z}\R, \quad a \longmapsto a\otimes 1$ sends generators of the group $H_2(M)$ to a basis.
\item  The kernel of the map $\iota_{\R}$ is precisely the torsion subgroup of $H_2(M)$. Let $\tau(H_2(M))$ be the torsion subgroup of $H_2(M)$, we will use the natural isomorphism $$H_2(M)\otimes_{\Z}\R \cong \dfrac{H_2(M)}{\tau H_2(M)}\otimes_{\Z}\R$$
\item  The image of $\iota_{\R}(H_2(M))\subset H_2(M)\otimes_{\Z}\R$  is discrete. Indeed, from the previous item, it is equal to the image of of the corresponding map $H_2(M) /\tau H_2(M)\longrightarrow H_2(M)/\tau H_2(M) \otimes_{\Z}\R$ which sends the finitely generated abelian group $H_2(M)/\tau H_2(M)$ to the lattice of integer homology classes.
\end{enumerate}

\subsection{The de Rham form}\label{sec:dRform}

Here we discuss how the familiar de Rham isomorphism can be interpreted as a closed 2-form $\Theta \in \Omega^2(M;H_2(M;\R))$.
 This form will be the curvature of the de Rham algebroid in \S \ref{sec:dRalgd}. Note that the construction of the de Rham form we give here holds for an arbitrary manifold $M$. Whence, just for the purposes of the current section, we will not need assumptions \ref{asm}. However, we can prove that this algebroid is integrable if $M$ is simply connected, as we explain in \S \ref{sec:dRalgd}.
  
Let us begin with the natural map 
\begin{equation}\label{eqn:Jmap}
\mathcal{J}_{M}:H_2(M,\mathbb{R}) \longrightarrow \operatorname{Hom}(H^2_{\operatorname{dR}}(M),\R) \quad \mathcal{J}_{M}([\gamma])(\omega):=\int_{[\gamma]}\omega
 \end{equation}
  \begin{remark}\label{Jdescription}
Using the de Rham isomorphism, we may identify $\operatorname{Hom}(H^2_{\operatorname{dR}}(M);\R)$ with the bidual  $H_2(M,\R)^{**}$. Then the map $\mathcal{J}_{M}$ is given by the bidual embedding (isomorphism) $\mathcal{J}_{M}{[\gamma]}=\operatorname{ev}([\gamma]) $. 
\end{remark}
In the following pages, as in the above remark, we will use freely the two identifications:
\begin{enumerate}
\item $\operatorname{dR}:H^2_{dR}(M)\longrightarrow \operatorname{Hom}(H_2(M;\R),\R)$.
\item $\operatorname{ev}:H_2(M;\R) \longrightarrow H_2(M;\R)^{**}$.
\end{enumerate}

Then:
\begin{itemize}
\item We have defined an element $\mathcal{J}_{M} \in \operatorname{Hom}\Big{(}H_2(M,\mathbb{R});  H_2(M;\R)^{**}  \Big{)}\cong \operatorname{Hom}(H_2(M;\R); H_2(M;\R))$; the de Rham theorem implies that $\mathcal{J}_{M}$ defines a cohomology class, that we also denote $\mathcal{J}_M$, which belongs to 
$ H^2_{\textrm{dR}}\Big{(}M; (H_2(M;\R)  \Big{)}$. This class is represented by some closed form $$\Theta \in \Omega^2(M, H_2(M;\R)).$$ 
Here the notation means that $\Theta$ is valued in the trivial bundle $M \times H_2(M;\R)$.
\item The class $[\Theta]$ can itself be integrated over cycles. For every $[\sigma] \in H_2(M;\mathbb{R})$ we have $$\int_{[\sigma]}[\Theta] \in \operatorname{Hom}(H^2_{\operatorname{dR}}(M);\R).$$ We know the values of these integrals because, by the de Rham Theorem, they characterize $[\Theta]$. Thus: 
$\int_{[\sigma]}[\Theta]=\mathcal{J}_{M}([\sigma])$. In other words
\begin{equation}\label{integralidentity}\bigg{(}\int_{[\sigma]}[\Theta]\bigg{)}([\omega])=\int_{[\sigma]}[\omega],\end{equation}
for every $[\sigma]\in H_2(M,\mathbb{R})$ and $[\omega] \in H^2_{\operatorname{dR}}(M)$.
By the linearity of the integral this also means, for every $[\omega]$:
$$
 \bigg{(}\int_{[\sigma]}[\Theta]\bigg{)}(\omega)=\int_{[\sigma]}\Big{(}\Theta([\omega])\Big{)}.$$
  Looking at $\Theta$ as a form with values in $H_2(M;\R)$ these identities become tautological:
 \begin{equation}\label{taut}
 \int_{[\gamma]}\Theta = [\gamma], \quad [\gamma]\in H_{2}(M;\R).\end{equation}
\end{itemize}


\subsection{Construction of the de Rham algebroid}\label{sec:dRalgd}

Now let us show how we can use the de Rham form given in \S \ref{sec:dRform} to construct an integrable Lie algebroid. Let us give the following:

\begin{proposition-definition}\label{prodef:dRalgd}
The form $\Theta \in \Omega^2(M;H_2(M))$  gives rise to a trivially abelian transitive Lie algebroid, $$B_\Theta^{dR}=TM \oplus_{\Theta } ( M \times H_2(M;\R))$$  which is unique up to isomorphism. The associated isomorphism class is called the \emph{de Rham algebroid} associated with $M$.
\end{proposition-definition}
Regarding the Lie algebroid structure of $B_{\Theta}^{dR}$, recall that the bracket of two sections is $$[(X,a),(Y,b)]=([X,Y],X(b)-Y(a)-\Theta(X,Y))$$ Here $a,b \in C^\infty(M, H_2(M;\R))$ and $X,Y$ are vector fields. The anchor map is the projection to $TM$.

\begin{thm}\label{thm:crucial}
Let $M$ be simply connected,
then the algebroid $B^{dR}_{\Theta}$ is integrable.
\end{thm}
\begin{proof} Since $M$ is simply connected, the monodromy map $\partial $ is given by integrals of homology cycles in $M$, namely by the map $[\gamma] \longmapsto \int_{[\gamma]} \Theta \in H_2(M;\R)$. Let $H:\pi_2(M) \longrightarrow H_2(M)$ be the Hurewicz map (in this case it is an isomorphism). Then equation \eqref{taut} (with all the identifications we made) says that the monodromy map is given exactly by the composition:
$$\xymatrix{{\pi_2(M)}\ar[r]& \pi_2(M)\otimes_{\mathbb{Z}}\R \ar[r]^{H \otimes \operatorname{Id}}& H_2(M;\R)   },$$ which we have already observed to have discrete image.
\end{proof}

\begin{remark}\label{rmk:crucial2}
In the proof of theorem \ref{thm:crucial}, we considered the extension of the Hurewicz morphism to $\pi_2(M)\otimes_{\Z}\R$. Note that the vector space $\pi_2(M)\otimes_{\Z}\R$ ``kills'' all the torsion elements of $\pi_2(M)$. In other words, torsion elements as such lie in the kernel of the map $\pi_2(M) \to \pi_2(M)\otimes_{\Z}\R$. This is consistent with the fact that the boundary map $\partial$ also ``kills'' the torsion elements, as we mentioned in \S \ref{sec:CFobstr} (item (b)).
\end{remark}

\subsection{The lifting construction}\label{sec:lift}
Let $M$ a simply connected manifold such that $\pi_2(M)$ is finitely generated and $V$ a vector space (finite dimensional). We construct  a canonical extension of the trivially abelian transitive algebroid 
$$\xymatrix@1{0 \ar[r]& M \times V \ar[r]& B_{\omega} \ar[r]& TM \ar[r]& 0}$$ where $B_{\omega}$ is the Lie algebroid $TM\oplus_{\omega}(M\times V)$ defined by a closed form $\omega \in \Omega^2(M;V)$. Recall that the Lie bracket of $B_{\omega}$ is given by $$[X\oplus\xi,Y\oplus\eta]=[X,Y]\oplus\{X(\eta)-Y(\xi)-\omega(X,Y)\}$$

Let $\Theta$ be a representative of the class $\cJ_{M}$ as above. We have the associated de Rham algebroid $B_{\Theta}^{\operatorname{dR}}$ which is integrable. Consider the closed form $\Theta \oplus \omega \in \Omega^2(M;H_2(M;\R) \oplus V)$ and the algebroid 
$$B_{\Theta\oplus\omega}= TM\oplus_{\Theta\oplus\omega} [(H_2(M;\R) \oplus V) \times M ]$$ The trivial bundle $(H_2(M;\R) \oplus V) \times M$ with the standard trivial connection and the Lie bracket of $B_{\Theta\oplus\omega}$ is 
\begin{equation*}\begin{split}
[(X,a\oplus\xi),(Y,b\oplus\eta)] & = [X,Y]\oplus\{X(b\oplus\eta)-Y(a\oplus\xi)-\widehat{\omega}(X,Y)\} \\ & = [X,Y]\oplus\{(X(b)-Y(a)-\Theta(X,Y))\oplus(X(\eta)-Y(\xi)-\omega(X,Y))\}\end{split}\end{equation*}

It is easy to see that the integrability of $B^{dR}_{\Theta}$ implies that $B_{\Theta\oplus\omega}$ is integrable as well.

Now consider the projection $p_V : H_2(M;\R) \oplus V \longrightarrow V$ and put $q_{\omega} = \id_{TM} \oplus p_V$. This is a surjective morphism of vector bundles $q_{\omega} : B_{\Theta\oplus\omega} \to B$:  It is a straightforward computation that, at the level of sections, it preserves the Lie brackets:
 \begin{equation}\begin{split}\label{morphismproperty}
 p_V [X \oplus a \oplus \xi, Y\oplus b \oplus \eta  ]&= p_V  [X,Y]\oplus
 \{(X(b)-Y(a)-\Theta(X,Y))\oplus(X(\eta)-Y(\xi)-\omega(X,Y))\} \\ &=[X,Y]\oplus (X(\eta)-Y(\xi) -\omega(X,Y)) \\&= [p_V (X \oplus a \oplus \xi ), p_V(Y \oplus b \oplus \eta) ].\end{split}
 \end{equation}
  We have proven:
\begin{thm}
Let $B_{\omega}$ be a trivally abelian Lie algebroid over a simply connected manifold $M$. The map $q_\omega$ gives an abelian extension $$\xymatrix{0 \ar[r]&H_2(M;\R) \times M \ar[r]&B_{\Theta\oplus\omega}\ar[r]^{q_{\omega}}&B_{\omega}\ar[r]&0}$$ such that $B_{\Theta\oplus\omega}$ is integrable. 
\end{thm}

\subsection{Examples}\label{sec:max}
\begin{ex}\label{molino}
Let us look again at the Almeida-Molino algebroid discussed in \ref{ex:AlmMol}. Recall that $M=\mathbb{S}^2\times 	\mathbb{S}^2$ equipped with the two projections $\operatorname{pr}_i:M \longrightarrow \mathbb{S}^2$. Let $V \in \Omega^2(\mathbb{S}^2)$ be the standard volume form and put $\omega_i:= \operatorname{pr}_i^*(V) \in \Omega^2(M)$. Then, for $\lambda\in \mathbb{R}$ consider the algebroid $B_\lambda$ defined by the closed $2$-form $$\alpha_\lambda= \omega_1+\lambda \omega_2.$$ Note that the dimension mentioned in Prop \ref{prop:irrdeg} is $n(B_{\lambda})=2$.
The monodromy map can be identified with the map $\mathbb{Z}^2 \ni (a,b) \longmapsto a+\lambda b$. Whence if $\lambda$ is irrational, $B_\lambda$ is not integrable.
To construct $B_\Theta^{dR}$, we have to look at $\operatorname{Hom}(H^2_{\operatorname{dR}}(M),\mathbb{R})=H^2_{\operatorname{dR}}(M)^*$. This can be identified with $H^2_{\operatorname{dR}}(M)$ itself by Poincar\'e duality if we
 define $\xi_1,\xi_2 \in \operatorname{Hom}(H^2_{\operatorname{dR}}(M),\R)$ to be the functionals
$$\xi_1(\eta)= \int_M \omega_2 \wedge \eta,\quad  \xi_2(\eta)= \int_M \omega_1 \wedge \eta.$$
Thus $(\xi_1,\xi_2)$  is the dual basis of $(\omega_1,\omega_2)$ as a basis of $H^2_{\operatorname{dR}}(M)$. It is readily checked that the closed form
$$\Theta:= \omega_1 \otimes \xi_1+ \omega_2 \otimes \xi_2,$$ as an element of $\Omega^2(M) \otimes (H^2_{\operatorname{dR}}(M))^*$ represents here our element $\mathcal{J}_M$.  Writing every representative $\eta$ of a class in $ H^2_{\operatorname{dR}}(M)$ as $\eta=\xi_1(\eta)\omega_1+ \xi_2(\eta)\omega_2$ we have:
$$\int_{[\sigma]}(\Theta)(\eta)=\int_{[\sigma]}\eta$$
This is the fundamental property \eqref{integralidentity} of $[\Theta]$.

With respect to the basis $(\xi_1,\xi_2)$ we can write the universal algebroid as $B_{\Theta}^{dR}=TM \oplus (M \times \R^{2})$ and the 
extension can be written as:
\begin{equation}\label{exampleextension}
TM \oplus_{(\Theta,\alpha_\lambda)}\R^3 \longrightarrow TM \oplus (M\times \R), \quad  (x,y,z) \mapsto x+\lambda y.
\end{equation}
\end{ex}

\begin{remark}
In the previous example we see that the de Rham lift may add more dimensions than necessary to build an integrable lift. This is immediately seen by comparing the extension \eqref{exampleextension} with the Almeida-Molino integrable lift we constructed in \S \ref{sec:AlmMol} (which adds only one dimension). Of course this is due to the \emph{universal character} of $B_\Theta^{dR}$ which gives rise to an integrable lift for every transitive Lie algebroid $B \to M$. 
\end{remark}

 

 \section{The non simply connected case} \label{sec:nonsc}
 
In this section we discuss the application of the de Rham lifting construction in case the base manifold $M$ is not necessarily simply connected. As mentioned in the introduction, the extra difficulty arises from the failure of the Hurewicz map to be an isomorphism. To deal with this, we pull back the algebroid $B \to M$ to the universal cover $\widetilde{M}$. Then we impose a certain topological assumption which somehow measure how the equivariant cohomology of $\widetilde{M}$ is \emph{mixed} with its non equivariant cohomology. Indeed, the curvature of $B$ lies in the equivariant cohomology but the obstruction, i.e. the integral of forms is done with respect to the standard homology.
This assumption 
 ensures that the de Rham integrable lift of the pullback algebroid  passes to an integrable lift of the original algebroid $B$.

We start by explaining the motivation behind this assumption in \S \ref{sec:assnonsc}. In \S \ref{sec:solnonsc} we explain this assumption solves the problem.

\subsection{The assumption}\label{sec:assnonsc}

Let us make a fresh start with a manifold $M$ such that $\pi_2(M)$ is finitely generated.  Let $p : \widetilde{M} \to M$ be the universal covering map. Since $\pi_2(M)\cong \pi_2(\widetilde{M})$, applying the Hurewicz theorem on $\widetilde{M}$ we see that $H^2_{dR}(\widetilde{M};\R)$ is finitely generated. Put $\Gamma$ the fundamental group $\pi_1(M)$; it acts (say on the right) on $\widetilde{M}$, thus on (the left) on differential forms and singular cycles. These actions descends to homology and cohomology. Let $Z^2(\widetilde{M};\R)$ be the vector space of closed two-forms on $\widetilde{M}$.

Our assumption arises from analyzing the equivariance of the de Rham form $\Theta \in Z^2\big{(}\widetilde{M}; H^2_{dR}(\widetilde{M})^*)$, which is a representative of the de Rham isomorphism $\cJ_{\widetilde{M}}$ as a class in $H^2_{dR}(\widetilde{M};H_2(\widetilde{M},\R)) = H^2_{dR}(\widetilde{M};H^2_{dR}(\widetilde{M})^{\ast})$.

\begin{prop}\label{identificationsections} 
We have a canonical identification
$$\Big{\{} \textrm{linear sections S of  the surjection } q:Z^2(\widetilde{M};\R) \longrightarrow H^2(\widetilde{M})  \Big{\}}  \longleftrightarrow  \Big{\{}\textrm{ representatives }\Theta \textrm{ of } \cJ_{\widetilde{M}}  \Big{\}} . $$ This is defined in both directions by the formula  
\begin{equation}\label{eq:correspondence}
\langle \Theta(X,Y) , [\omega] \rangle = S([\omega])(X,Y) 
\end{equation}
where $S:H^2(\widetilde{M};\R) \longrightarrow Z^2(\widetilde{M};\R)$ is a section and $X,Y$ are vector fields.
\end{prop}
\begin{proof}
Given a linear section $S$, formula \eqref{eq:correspondence} defines a differential form $\Theta_S$ which is closed because $ \langle d\Theta_S , [\omega] \rangle = d ( S([\omega])) =0.$ Integrating \eqref{eq:correspondence} over homology classes we also see that $\Theta_S$ represents $\mathcal{J}_{\widetilde{M}}$ because $S$ is a section. 

Conversely, assume that a representative $\Theta$ of $\cJ_{\widetilde{M}}$ is given and define $S_\Theta$ using \eqref{eq:correspondence}. This is a section because, for every cohomology class $[\omega]$ and every homology class $[\sigma]$: $$\int_{[\sigma]}[\omega]=  \langle \int_{[\sigma]}\Theta_S, [\omega] \rangle =\int_{[\sigma]}S([\omega]).$$ Whence $S([\omega]) $ and $\omega$ give same cohomology classes.
\end{proof}

Now assume that a representative $\Theta$ of $\cJ_{\widetilde{M}}$ is equivariant. Then the corresponding section $S_\Theta$ is also equivariant. Moreover it is an injective map $S_\Theta: H^2(\widetilde{M};\R) \longrightarrow Z^2(\widetilde{M};\R)$. We obtain a vector space of differential forms  $$\mathcal{E}= \operatorname{Image}(S_\Theta)$$ The equivariance of $S_{\Theta}$ implies that $\cE$ has the properties given in the following assumptions:

 \begin{assumptions}\label{ass:nonsc}
 There exists a finite dimensional subspace of forms $\mathcal{E} \subset Z^2(\widetilde{M};\R)$ such that
\begin{enumerate}
\item	it is $\Gamma$-invariant as a subspace ($ \gamma \cdot \varphi  \in \mathcal{E}$ for every $\varphi \in \mathcal{E}$, $\gamma \in \Gamma$).
\item The natural map $r:\mathcal{E} \longrightarrow H^2_{dR}(\widetilde{M};\R)$ is surjective.
\end{enumerate}
 \end{assumptions}

We denote  $\langle \cdot , \cdot \rangle$ the natural duality $ \mathcal{E}^* \times \mathcal{E} \to \R$.

\subsection{The integrable lift}\label{sec:solnonsc}

Now let us consider an abelian transitive Lie algebroid $B_{\omega} \to M$. Here we show that with a space of forms $\cE$ as in assumption \ref{ass:nonsc} we can construct explicitly an integrable lift of $B_{\omega}$. Let us start by assuming \ref{ass:nonsc}. Using this asumption we first construct an equivariant closed form $\Theta$ which is not necessarily a representative of the de Rham element $\cJ_{\widetilde{M}}$.
 
 Proceeding as before, the de Rham theorem identifies $\operatorname{Hom}\Big{(}\mathcal{E},  H^2_{dR}(\widetilde{M},\R) \Big{)}$ with $H^2_{dR}(\widetilde{M};\mathcal{E}^*)$. The map $r$ is then represented by a closed form 
 $$\Theta \in \Omega^2(\widetilde{M};\mathcal{E}^*).$$ 
 
 Integrating $\Theta$ on some class $[\sigma]\in H_2(\widetilde{M})$ produces a functional:  $\int_{[\sigma]}\Theta \in \mathcal{E}^*$ for which is, by definition: 
 $$\big{\langle} \int_{[\gamma]} \Theta ,\varphi \big{\rangle} = \int_{[\gamma]} \Theta(  \varphi ).$$  More importantly, by the de Rham Theorem
 $$
 \big{\langle} \int_{[\sigma]}\Theta , \varphi \big{\rangle} = \int_{[\sigma]} r(\varphi) = \int_{[\sigma]} [\varphi], \quad \varphi \in \mathcal{E}.$$
This last identity suggests that, in this case, we have a canonical (tautological) representative. Indeed we can compute the form $\Theta$. For $v,w \in T_m\widetilde{M}$ and $\varphi \in \mathcal{E}$:
\begin{equation}\label{explicittheta}\langle \Theta(v,w), \varphi \rangle = \varphi(v,w).\end{equation}

Regarding the equivariance of $\Theta$, let $R_{\gamma}: \widetilde{M} \longrightarrow \widetilde{M}$ be the right action, then $\Gamma$ acts on the left on $\Omega^2(\widetilde{M};\mathcal{E}^*)$ according to the formula
 $$ \langle \omega(X,Y),\varphi \rangle = \langle \omega(dR_{\gamma}(X), dR_{\gamma}(Y), \varphi \cdot \gamma^{-1} \rangle, \quad   X,Y \in \Gamma(\widetilde{M}), \quad \varphi \in \mathcal{E}.$$  
With a slight abuse we say that a form is equivariant when is invariant under the previous action. Identity \eqref{explicittheta} immediately implies:
\begin{prop}
The form $\Theta$ is equivariant.	
\end{prop}
It follows that $\Theta$ descends to a form on $M$ with coefficients on the flat bundle ${E}^*:= \widetilde{M} \times_{\Gamma} \mathcal{E}^*$ over $M$, which corresponds to $\mathcal{E}^*$.
We call this form $\theta \in \Omega^2(M;E^*)$ and we use it to construct the algebroid:
 $$B^{dR}_\theta:=TM \oplus_\theta E^*.$$
Let us describe its structure: call $\nabla^{E^*}$ the flat connection in $E^*$.  Under the description which identifies sections of $E^*$ with equivariant maps $\widetilde{M} \longrightarrow \mathcal{E}^*$ this corresponds to the standard differential. For this reason $\nabla^{E^*} \theta=0$ and the bracket is
 $$[X \oplus \xi, Y \oplus \eta]=[X,Y] \oplus (\nabla_X^{E^*}\eta - \nabla_Y^{E^*}\xi -\theta(X,Y)), \quad X,Y \in \Gamma(TM), \,\, \xi,\eta \in \Gamma(E^*).$$

 \begin{thm}
 The algebroid 	$B^{dR}_\theta$ is integrable.
 \end{thm}
 \begin{proof}
 Let us call $\partial:\pi_2(\widetilde{M}) \longrightarrow \mathcal{E}^*$, the monodromy map of $B_\theta^{dR}$. We are assuming the natural map $r:\mathcal{E}\longrightarrow H^2_{dR} (\widetilde{M};\R)$ to be surjective. This implies that the transpose map $r^*: \operatorname{Hom}(H^2_{dR}(\widetilde{M},\R);\R) \longrightarrow \mathcal{E}^*$ is injective. Using the de Rham theorem and the Hurewicz isomorphism to identify the domain of $r^*$ with $\pi_2(\widetilde{M})\otimes \R$ we see immediately that $r^*= \partial \otimes 1_{\R}$. It follows that the image of the monodromy map is discrete. \end{proof}
Now recall that the original Lie algebroid $B_{\omega} \to M$ can be written as an abelian extension 
$$0 \to \widetilde{M} \times_{\Gamma} V \to TM \oplus_{\omega}(\widetilde{M} \times_{\Gamma} V) \to TM \to 0$$ Exactly as in \S \ref{sec:lift} we consider the form 
$\hat{\omega} = \theta \oplus \omega \in Z^2(M,E^* \oplus \widetilde{M} \times_{\Gamma} V)$ and the algebroid 
$$B_{\hat{\omega}}=TM \oplus E^* \oplus (\widetilde{M}\times_{\Gamma}V).$$ 
On the bundle $E^* \oplus  \widetilde{M}\times_{\Gamma}V$ we are considering the direct sum flat connection 
$\widehat{\nabla}:=\nabla^{E^*}\oplus \nabla$
where we call just $\nabla$ the flat connection on 
$\widetilde{M} \times_{\Gamma}V$. In particular 
$\widehat{\nabla}\hat{\omega}=0$ and its bracket is in  normal form given by $\widehat{\nabla}$ and $\hat{\omega}$.

It is easy to see that 
$E^* \oplus  \widetilde{M}\times_{\Gamma}V$ is integrable as well. A computation identical to the \eqref{morphismproperty} shows that the projection $p_V:E^* \oplus \widetilde{M} \times_{\Gamma} V \to \widetilde{M} \times_{\Gamma} V$ is a morphism of algebroids. 
We have thus proven:
\begin{thm}
	Under the assumptions \ref{ass:nonsc} we have the extension of transitive algebroids
	$$0 \to E^* \to B_{\hat{\omega}} \to B_{\omega} \to 0$$ with $B_{\hat{\omega}}$ integrable.  
\end{thm}

\begin{remark}
In the overall we proved that a vector subspace $\cE$ of closed 2-forms in $Z^2(\widetilde{M},\R)$ which satisfies assumptions \ref{ass:nonsc} exists if and only if there exists an equivariant form $\Theta \in \Omega^2(\widetilde{M},\cE^{\ast})$ such that $B_{\theta}^{dR}$ is integrable. Also notice that the identification in Proposition \ref{identificationsections} restricts to 
$${\{} \textrm{equivariant sections}  {\}}  \longleftrightarrow {\{} \textrm{equivariant forms}{\}}.$$ In some cases an equivariant section can be constructed, for instance if the fundamental group is finite, by averaging.  
\end{remark}

\bibliographystyle{plain}

\begin{thebibliography}{10}

\bibitem{AlmeidaMolino}
Rui Almeida and Pierre Molino.
\newblock Suites d'{A}tiyah et feuilletages transversalement complets.
\newblock {\em C. R. Acad. Sci. Paris S\'er. I Math.}, 300(1):13--15, 1985.

\bibitem{AZ3}
Iakovos Androulidakis and Marco Zambon.
\newblock Almost regular {P}oisson manifolds and their holonomy groupoids.
\newblock {\em Selecta Math. (N.S.)}, 23(3):2291--2330, 2017.

\bibitem{CatContr}
Alberto~S. Cattaneo and Ivan Contreras.
\newblock Relational symplectic groupoids.
\newblock {\em Lett. Math. Phys.}, 105(5):723--767, 2015.

\bibitem{CatFeld}
Alberto~S. Cattaneo and Giovanni Felder.
\newblock Poisson sigma models and symplectic groupoids.
\newblock In {\em Quantization of singular symplectic quotients}, volume 198 of
  {\em Progr. Math.}, pages 61--93. Birkh\"auser, Basel, 2001.

\bibitem{Crainic:2004aa}
Marius Crainic et~al.
\newblock Prequantization and {L}ie brackets.
\newblock {\em Journal of Symplectic Geometry}, 2(4):579--602, 2004.

\bibitem{CFLie}
Marius Crainic and Rui~Loja Fernandes.
\newblock Integrability of {L}ie brackets.
\newblock {\em Ann. of Math. (2)}, 157(2):575--620, 2003.

\bibitem{DuistKolk}
J.~J. Duistermaat and J.~A.~C. Kolk.
\newblock {\em Lie groups}.
\newblock Universitext. Springer-Verlag, Berlin, 2000.

\bibitem{FackSkandalis}
T.~Fack and G.~Skandalis.
\newblock Sur les repr\'esentations et id\'eaux de la {$C^{\ast} $}-alg\`ebre
  d'un feuilletage.
\newblock {\em J. Operator Theory}, 8(1):95--129, 1982.

\bibitem{Iglesias:1995aa}
Patrick Igl\'esias.
\newblock La trilogie du moment.
\newblock {\em Ann. Inst. Fourier (Grenoble)}, 45(3):825--857, 1995.

\bibitem{MK2}
Kirill C~H Mackenzie.
\newblock {\em General theory of Lie groupoids and Lie algebroids}, volume 213
  of {\em London Mathematical Society Lecture Note Series}.
\newblock Cambridge University Press, Cambridge, 2005.

\bibitem{Malcev45}
A.~Malcev.
\newblock On the theory of the {L}ie groups in the large.
\newblock {\em Rec. Math. [Mat. Sbornik] N. S.}, 16(58):163--190, 1945.

\bibitem{TsengZhu}
Hsian-Hua Tseng and Chenchang Zhu.
\newblock Integrating {L}ie algebroids via stacks.
\newblock {\em Compos. Math.}, 142(1):251--270, 2006.

\end{thebibliography}

\end{document}